\documentclass[letterpaper, 10 pt, conference]{ieeeconf}  

\IEEEoverridecommandlockouts

\usepackage{cite}
\usepackage{amsmath,amssymb,amsfonts}
\usepackage{algorithmic}
\usepackage{graphicx}
\usepackage{textcomp}
\usepackage{xcolor}
\usepackage{comment}
\usepackage[ruled]{algorithm2e}

\usepackage{amsthm}

\usepackage{orcidlink}
\usepackage[normalem]{ulem}

\newtheorem{theorem}{Theorem}
\newtheorem*{problem}{Problem}
\newtheorem{remark}{Remark}
\newtheorem*{remark*}{Remark}
\newtheorem{lemma}{Lemma}
\newtheorem*{lemma*}{Lemma}
\newtheorem{assumption}{Assumption}

\DeclareMathOperator*{\argmin}{arg\,min}

\DeclareMathOperator*{\tr}{tr}
\DeclareMathOperator*{\interior}{int}

\newcommand{\des}{\textrm{des}}

\title{\LARGE \bf
Adaptive Incentive Design with Regret Minimization
}

\author{
Georgios Vasileiou\orcidlink{0009-0002-3679-0510}$^{\dagger}$,
Lantian Zhang\orcidlink{0000-0002-1814-5596}$^{\dagger}$,
and
Silun Zhang\orcidlink{0000-0003-3772-761X}$^{\dagger}$%
\thanks{This work has been partially supported by the Wallenberg AI, Autonomous Systems and Software Program (WASP), funded by the Knut and Alice Wallenberg Foundation.}%
\thanks{$^{\dagger}$ Georgios Vasileiou, Lantian Zhang and Silun Zhang are with the Dept. of Mathematics at the KTH Royal Institute of Technology, Stockholm, 10044, Sweden. {\tt\small\{geovas, lantian, silunz\}@kth.se}}%
}

\begin{document}

\maketitle
\thispagestyle{empty}
\pagestyle{empty}
\bstctlcite{IEEE:BSTurloff} 
\begin{abstract}
Incentive design constitutes a foundational paradigm for influencing the behavior of strategic agents, wherein a system planner (principal) publicly commits to an incentive mechanism designed to align individual objectives with collective social welfare. This paper introduces the Regret-Minimizing Adaptive Incentive Design (RAID) problem, which aims to synthesize incentive laws
under information asymmetry and achieve asymptotically minimal regret compared to an oracle with full information.
To this end, we develop the RAID algorithm, which employs a switching policy alternating between probing (exploration) and estimate-based incentivization (exploitation). The associated type estimator
relies only on a weaker excitation condition required for strong consistency in least squares estimation, substantially relaxing the persistence-of-excitation assumptions previously used in adaptive incentive design.
In addition, we establish the strong consistency of the proposed type estimator and prove that the incentive obtained asymptotically minimizes the planner's average regret almost surely. Numerical experiments illustrate the  convergence rate of the proposed methodology.
\end{abstract}


\section{Introduction}

Incentivization is a central topic in many economic and engineered systems, where a central planner seeks to coordinate the behavior of agents with unknown or misaligned objectives.
Designing incentives that align individual interests with collective goals gives rise to the general class of \textit{Incentive Design} problems, also known as principal–agent or reverse Stackelberg games~\cite{grootReverseStackelbergGames2012}. 
These problems describe sequential decision-makers coupled through interdependent costs, where a system planner (leader) aims to steer
a population of self-interested agents (followers)
toward a globally favorable outcome.
Each agent, however, acts to minimize its private cost that depends on its own action and the principal's posted incentive. 
The planner must
design and publicly commit to an incentive law which maps agents' actions to rewards or penalties, thereby shaping the agents' strategic behavior.
From a control-theoretic perspective~\cite{hoControltheoreticViewIncentives1980, basarAffineIncentiveSchemes1984, ratliffPerspectiveIncentiveDesign2019}, such a commitment can be viewed as a feedback mechanism that prescribes a desired equilibrium solution among agents.
Problems of this type often emerge in adaptive pricing of distributed energy resources in smart grids~\cite{liDistributedOnlinePricing,liSociallyOptimalEnergy2024, samadiDemandSideMagement2012}, congestion-aware road tolling~\cite{povedaDistributedAdaptivePricing, barreraDynamicIncentives, grootSystemOptimalRoutingTraffic2015}, and data-crowdsourcing in federated learning markets~\cite{hoAdaptiveContractDesignCrowdsourcing, dingContractDesignFederated2021}. In such applications, the planner must adapt incentives under uncertainty about agents' preferences, making adaptive methods essential.

A significant challenge in incentive design problems is  adverse selection~\cite{picardDesignIncentiveSchemes1987}, i.e., information asymmetry arising from the planner’s incomplete knowledge of the agents’ private cost functions or preferences. Historically, adverse selection has been addressed via the design of mechanisms that induce truthful participation~\cite{nisanAlgorithmicGameTheory2007, hartlineMechanismDesignApproximation2013}, methods which are restricted to static, one-shot interactions and are strongly model-dependent. 
To overcome these limitations,  recent \textit{adaptive incentive design}~\cite{ratliffAdaptiveIncentiveDesign2021} approaches infer agents' preferences through repeated principal-agent interactions and the information derived from the corresponding equilibrium outcomes. 
Relaxing the exactness of agent responses, Chen et al.~\cite{chenActiveInverseMethods2025} study Stackelberg settings where agents exhibit bounded rationality, and the leader actively designs informative actions for a maximum-likelihood estimator.
Departing from estimation-based methods, Maheshwari et al.~\cite{maheshwariAdaptiveIncentiveDesign2024} study incentive design under a two-timescale framework, where agents are learning individuals that exponentially converge to a Nash equilibrium. The planner leverages agent externalities, the difference between social objective and the agent cost gradients, to design incentives on a slower timescale. For finite action spaces, \cite{yorulmazSoftInducementFramework2025} examines equilibrium steering in Bayesian normal-form games with bimatrix payoffs, where the principal guides learning agents toward desired action profiles by imposing constant, lower-bounded incentives, which achieves sublinear regret in utility. 

In this work, we introduce the Regret-Minimizing Adaptive Incentive Design (RAID) problem, which integrates online type estimation together with incentive design to achieve asymptotically vanishing behavioral regret relative to an oracle. The paper makes two main contributions.

First, we investigate 
the dual problem of type estimation and incentive commitment, examining how online parameter estimation affects the planner’s ability to regulate agent behavior.
Motivated by literature on adaptive estimation and control~\cite{tzelaiExtendedLeastSquares1986}, we formulate the planner's regret of a given sequence of incentives as the cumulative tracking error up to iteration $t$. We propose an adaptive incentive policy
(Algorithm~\ref{alg:one}) that switches between probing (exploratory) and estimate-based (exploitative) incentives, and prove that the resulting incentive policy achieves almost-surely $O(t^\gamma \log t)$ regret, for $\gamma \in [ \tfrac{2}{3}, 1)$. The sublinear regret implies the policy is asymptotically optimal and balances the competing demands of learning and control in adaptive incentive design.

Second, we relax the excitation required for the planner's identification objective to a weak diminishing-excitation condition, known to guarantee the strong consistency of the least-squares estimator in stochastic approximation~\cite{laiLeastSquaresEstimates1982}. This result eliminates the standard persistence-of-excitation (PE) assumption made in adaptive incentive design, and provides strong consistency guarantees for the estimator with injection of diminishing excitation. 

Overall, our findings extend the adaptive-control perspective on adaptive incentive design by ensuring strongly consistent type estimation and providing almost-surely vanishing regret for an appropriate switching incentive policy.

\noindent\textbf{Notations.} Given some $n \in \mathbb{N}$, let $[n] = \{1, 2, \ldots, n\}$. 
For a vector-valued map $f:\mathbb{R} \rightarrow \mathbb{R}^n$,
denote derivatives $\nabla f = \big(\frac{\partial f_1}{\partial x}(x), \ldots, \frac{\partial f_n}{\partial x}(x)\big)^\top$
and $\nabla^2 f = \nabla (\nabla f)$. 
The notation $f(t) = O(g(t))$ denotes the existence of $c> 0$ and $t_0$ so that $|f(t)|\leq c|g(t)|$ for all $t\geq t_0$.
Likewise, $f(t) = \Omega(g(t))$ indicates $|f(t)|\geq c|g(t)|$.
We write $f(t) = \Theta(g(t))$ when both 
bounds hold, and $f(t) = o(g(t))$ if $ f(t) /g(t) \to 0$.
$C^k$ denotes the class of $k-$smooth functions.
 
\section{Problem Formulation}
\label{sec:problem_formulation}
Consider a noncooperative game between $n$ \textit{agents} and a single \emph{system planner}. Each agent $i\in [n]$ aims to minimize an individual cost function by choosing a strategy  $x_i \in \mathcal X_i$, where $\mathcal{X}_i\subset \mathbb R$ is a compact interval.\footnote{ Agents' strategies are scalar for clarity of exposition. Results extend directly to higher-dimensional strategies with notational changes only.}
Agent $i$ seeks  to minimize the individual cost
\begin{equation}
    \label{eq:costs}
    c_i(x_i,p_i) = \ell_i(x_i) + p_ix_i,
\end{equation}
where $\ell_i : \mathbb R \to \mathbb R$ is the agent's \emph{nominal cost}, and $p_i\in \mathbb R$ is an external incentive rate imposed by the system planner.
Here, the agent's nominal cost corresponds to an inherent preference over their strategy, and $p_ix_i$ is the planner's intervention, imposed to modify the nominal behavior. 
Linear incentives, as the simplest incentive mappings available to the planner, have been extensively examined in recent work on incentive design~\cite{maheshwariAdaptiveIncentiveDesign2024, liSociallyOptimalEnergy2024} and are sufficient to influence agent behavior by modifying marginal costs. 

The system planner's objective is to minimize a social cost $\Psi: \mathbb R^n \rightarrow \mathbb R$ that depends on the collective action $x = (x_1, \ldots, x_n)^\top$. When agents act in self-interest, 
aligning their behavior with the planner's objective requires that the planner design and implement an appropriate incentive. Define the set of socially optimal outcomes by
\[\mathcal{X}_\des = \argmin_{x\in \mathcal{X}} \Psi(x), \]
where  $\mathcal{X}=\prod_{i\in [n]} \mathcal{X}_i$.
If $\Psi(\cdot)$ is strongly convex then $\mathcal{X}_\des$ is a singleton; otherwise, the planner may select any $x_\des \in \mathcal{X}_\des$ as the desired agent profile. We restrict our attention to problems where $x_\des$ it interior to the feasible region.
\begin{assumption}
    \label{ass:feasibility}
    There exists a profile $x_\des \in \mathcal{X}_\des$ such that $x_\des \in \prod_{i\in [n]} \interior \mathcal{X}_i$.
\end{assumption}

A fundamental challenge for the planner is the \textit{information asymmetry} that arises when the agent's nominal costs $\ell_i$ are unknown. 
This information asymmetry prevents the direct computation of optimal incentives and necessitates the use of type identification methods. The planner parameterizes agents' nominal costs with 
\begin{equation}
\label{eq:parametrization}
\ell_i(x_i) = \theta_i^{*\top} \Phi(x_i),
\end{equation}
where $\Phi:\mathbb R \rightarrow \mathbb R^d$ is the monomial basis\footnote{For notational simplicity, we assume that all agent's nominal costs are parametrized with the same $d$-dimensional basis $\Phi$.}
\(\Phi(x) = (
    x,\, x^2,\,  \ldots,\, x^d)^\top, \)
and $\theta_i^* \in  \mathbb R^d$ is the private \textit{type parameter} of agent $i$.
The system planner must issue incentives that both facilitate the learning of agent types and regulate their behavior toward the intended profile $x_\des$.

\begin{assumption}
\label{ass:types}
    For every agent $i\in[n]$, there exist positive constants $m_i$ and $M_i$, such that $\theta_i^*$ belongs to a set of admissible types $\Theta_i$ that satisfies
    \[ \theta_i^* \in \Theta_i \triangleq \{\theta \in \mathbb R^d :  m_i \leq \theta^\top \nabla^2 \Phi(x_i) \leq  M_i, \, \forall x_i \in \mathcal{X}_i\}.\]
\end{assumption}
Under Assumption~\ref{ass:types}, each $\Theta_i$ is closed and convex, and the nominal cost $\ell_i$ of agent $i\in [n]$ is $m_i$-strongly convex and has an $M_i$-Lipschitz gradient on $\mathcal{X}_i$.

When the planner issues an arbitrary incentive $p_i \in \mathbb R$ to agent $i$, the agent's response is given by a \textit{best-response map} $x_i^*:\mathbb R \rightarrow \mathcal{X}_i$, which satisfies 
\begin{equation}
\label{eq:agent_full_response}
x_i^*(p_i) \in \argmin_{x\in \mathcal X_i} (\theta_i^{*\top} \Phi(x) +p_i x ).
\end{equation}
\begin{remark}
    Throughout this work, we assume that each agent immediately selects the best response according to \eqref{eq:agent_full_response} and do not consider the procedure by which the minimization problem is solved. 
    This is a
    standard (idealized) best-response model, often utilized in incentive design~\cite{liSociallyOptimalEnergy2024, ratliffAdaptiveIncentiveDesign2021}. 
\end{remark}
\begin{remark}
    Due to the decoupled structure of $\ell_i$ in \eqref{eq:costs}, agents' strategies are independent of other agents, as is evident in the best-response~\eqref{eq:agent_full_response}. This setting is often encountered in applications of incentive design, including demand-side energy management~\cite{liDistributedOnlinePricing, samadiDemandSideMagement2012} and federated learning data markets~\cite{dingContractDesignFederated2021}, where users' energy-scheduling  and participation cost overheads are locally determined. We will extend
    the framework introduced in this paper to include 
    multi-agent interactions in a following journal work.
\end{remark}

The next result characterizes the agent’s best-response map under Assumption~\ref{ass:types}, and establishes a bijection between incentives $p_i$ and responses $x_i^*(p_i)$, in an appropriate open subset of $\mathbb R$.
This result is developed from the case of unconstrained strategies presented in~\cite[Lemma~1]{liDistributedOnlinePricing}.
\begin{lemma}
\label{lemma:NE_bijection_mapping}
    Under Assumption~\ref{ass:types}, for each agent $i \in [n]$ and $\theta_i^* \in \Theta_i$, there is a
    $C^1$ map $h(p)$ such that the best-response map $x_i^*(p) = h(p)$. Over the open set $\mathcal{P}_i = h^{-1}(\interior \mathcal{X}_i)$, it holds that, for any $p\in \mathcal{P}_i$, 
    \begin{equation}
    \label{eq:first_order_condition}
    \theta_i^{*\top} \nabla \Phi(h(p)) + p =0, \text{ and}
    \end{equation}
    \begin{equation}
    \label{eq:bijection_derivative_bounded}
    -\frac{1}{ m_i} \leq h'(p) \leq -\frac{1}{ M_i},
    \end{equation}
\end{lemma}
\begin{proof}
Consider $x^*_i(p_i)$ to be a solution of \eqref{eq:agent_full_response} within $\interior \mathcal{X}_i$.
Then $x^*_i(p_i)$ satisfies the first-order optimality condition for $c_i(x_i,p_i)$, and therefore
\begin{equation}
    \label{eq:lemma1_eq1}
    -\theta_i^{*\top} \nabla\Phi(x^*_i(p_i)) + p_i =0.
\end{equation}

The mapping $f: x_i \mapsto -\theta_i^{*\top} \nabla \Phi(x_i)$ is $\mathcal{C}^1$ with a nonzero derivative so, by the inverse function theorem, there exists a $\mathcal{C}^1$ map $f^{-1}: -\theta_i^{*\top} \nabla\Phi(x_i) \mapsto x_i$. Due to \eqref{eq:lemma1_eq1}, $f^{-1}$ maps $p_i \mapsto x^*_i(p_i)$
and satisfies
\[(f^{-1})'(p_i) = \frac{1}{f'(x_i^*(p_i))} = \frac{-1}{\theta_i^{*\top}\nabla^2\Phi(x_i^*(p_i))}.\]
Define the open set $\mathcal{P}_i = f(\interior \mathcal{X})$.
For any  $p \in \mathcal{P}_i$ it holds that $\theta_i^{*\top}  \nabla  \Phi(f^{-1}(p)) + p= - f(f^{-1}(p)) + p  = 0$.
\end{proof}

The planner can only learn from agent responses that vary smoothly with incentives, so we introduce the notion of an \textit{informative region} $\mathcal{P}_i$, defined in Lemma~\ref{lemma:NE_bijection_mapping} for each agent $i\in [n]$. We refer to any $p \in \mathcal{P}_i$ as an informative incentive. 
\begin{remark}
    The set $\mathcal{P}_i$ of each agent depends on $\theta_i^*$, as evidenced by \eqref{eq:first_order_condition}, and is unknown to the system planner.
    However, the planner can identify whether an incentive is informative
    by the observation $x^*_i(p) \in \interior \mathcal{X}_i$.
\end{remark}

If the agents' preferences were known, the planner could elicit the desired response $x_\des = (x_{1,\des}, \ldots, x_{n,\des})^\top$ by issuing to $i\in [n]$ the incentive $p_{i,\des} = - \theta_i^{*\top} \nabla \Phi(x_{i,\des})$. 
In the absence of this information, the planner must construct an estimator $\hat \theta_i(t)$ of $\theta_i^*$ from incentive-response trajectories.  
Given estimate $\hat \theta_i(t)$, consider the adaptive incentive scheme
\begin{equation}
\label{eq:self_tuning_regulator}
p_i(t+1) = - \hat\theta_i(t)^\top \nabla \Phi(x_{i,\des}),
\end{equation}
which approximates the (unrealizable) regulator that assumes knowledge of $\theta_i^*$.

Denote the trajectory of player $i$ up to time $t$ as $\mathcal{T}_i(t) = \{(p_i(\tau), x_i(\tau))\}_{\tau \leq t}$, where $x_i(t) = x_i^*(p_i(t))$. Whenever $x_i(t) \in \interior \mathcal{X}_i$, the observation pair $(\hat p_i(t), x_i(t))$ 
satisfies the observation equation
\begin{equation}
\label{eq:model}
\hat p_i(t) 
= -\theta_i^{*\top} \nabla\Phi(x_i(t)) + e_i(t),
\end{equation}
where $e_i(t)$ is an unobservable noise.
We make the following assumption on the incentive–response noise.

\begin{assumption}
\label{ass:noise}
For each $i\in [n]$, the process $\{e_i(t)\}_{t\geq 1}$ in \eqref{eq:model} is i.i.d., zero-mean and satisfies $\sup_{t\geq 1}\mathbb E[|e_i(t)|^a] < \infty$ for $a > 2$. Moreover, $e_i(t)$ is independent of $x_i(t)$.
\end{assumption}
\begin{remark}
    The assumption that $e_i(t)$ is independent of $x_i(t)$ is standard in adaptive incentive design
    \cite{ratliffAdaptiveIncentiveDesign2021}. It reflects a setting in which type estimation and incentive design are performed by distinct entities, and only the noisy incentive signal is available for estimation.
    When independence does not hold, 
    \eqref{eq:model} becomes an Error-in-Variables (EIV) regression, whose analysis is deferred to a following journal version.
\end{remark}

We quantify the performance of a selected incentive sequence with respect to the cumulative tracking error from $x_\des$. Formally, define the $t$-stage regret of a given incentive sequence $\{p(\tau)\}_{\tau \leq t}$ to be
\begin{equation}
\label{eq:regret_definition}
\begin{aligned}
        R_t
        = \sum_{\tau = 1}^t \|x(\tau) - x_\des\|^2_2 = \sum_{\tau = 1}^t \|x^*(p(\tau)) - x_\des\|_2^2,
\end{aligned}
\end{equation}
where $p(t)$ and $x(t)$ denote $p(t) = (p_1(t),\, \ldots,\, p_n(t))^\top$, and  $x(t) = (x_1^*(p_1(t)),\, \ldots,\, x_n^*(p_n(t)))^\top$, respectively.

The planner's objective is twofold: to ensure accurate type estimation through sufficient exploration of the parameter space, and to regulate the collective behavior toward the social optimum. We formalize this task in the Regret-minimizing Adaptive Incentive Design (RAID) problem.

\begin{problem}[Regret--minimizing Adaptive Incentive Design -- RAID]
    Given a social cost function $\Psi(\cdot)$ and a desired agent profile $x_\des\in \argmin_{x\in \mathcal{X}} \Psi(x)$,
    a system planner with no prior knowledge of the true type 
    $\theta^* = (\theta_1^{*\top}, \ldots, \theta_n^{*\top})^\top$
    aims to design a type estimator 
    $\hat \theta(t)$ 
    and an incentive sequence $\{p(\tau)\}_{\tau \geq 1}$ satisfying two goals:
    \begin{enumerate}
        \item \emph{Strong consistency} of $ \hat \theta(t)$: For any initial estimation $\hat \theta(0)=\theta_0$, the estimator $\hat \theta(t)$ converges to $ \theta^*$ a.s., and
        \item \emph{Sublinear regret accumulation}: The $t$-stage average regret of the policy  $\{p(\tau)\}_{\tau \leq t}$, defined as
        in \eqref{eq:regret_definition},
        vanishes asymptotically a.s., that is,
        \(R_t = o(t) \text{ a.s. }\)
    \end{enumerate}
\end{problem}

Incentive Design can be seen as a Stackelberg game, wherein
the planner acts as a Stackelberg leader by committing to an incentive policy which affects followers' cost functionals and induces their best responses. In this way, each pair $(p_{\des}, x_{\des})$ chosen by the planner constitutes a Stackelberg equilibrium.
Achieving vanishing average regret corresponds to driving the incentive-response pair 
$(p(t),x(t))$ to the Stackelberg equilibrium $(p_\des, x_\des)$
in a mean-square sense. 
The construction of the type estimator to learn  $\theta^*$ and the conditions required for its strong consistency are presented in the next section.

\section{Strongly Consistent Type Estimation}
\label{sec:type_estimation}

In this section, we construct a type estimator and analyze its convergence under different levels of excitation present in the agents’ trajectories $\mathcal{T}_i(t)$.
Given a collection of incentive-response observations for each agent $i$ up to time $t$,
the estimator $\hat \theta_i(t)$ for $\theta_i^*$ can be constructed by
\begin{equation}
\label{eq:theta_i_minimization_problem}
    \hat \theta_i(t)  = \argmin_{\theta_i \in \Theta_i} \sum_{\substack{\tau \leq t: \\ 
    p_{i}(\tau) \in \mathcal{P}_i}} 
    |\theta_i^\top \nabla\Phi(x_i(\tau)) + \hat p_i(\tau)|^2.  
\end{equation}
Notice that agent response $x_i(t)$ satisfies \eqref{eq:model} only when the corresponding incentive $p_i(t)$ is informative, i.e., when $p_i(t)\in \mathcal{P}_i$. Due to Lemma~\ref{lemma:NE_bijection_mapping}, the event $\{p_i(t) \in \mathcal{P}_i\}$ in \eqref{eq:theta_i_minimization_problem} is equivalent to $\{x_i(t) \in \interior \mathcal{X}_i\}$,
which can be verified by the planner without requiring knowledge of $\theta^*$. 
Problem \eqref{eq:theta_i_minimization_problem} can be recursively solved by
\begin{subequations}
\label{eq:solution_estimator}
\begin{equation}
\label{eq:theta_i_solution_estimator}
\begin{aligned}
    \hat \theta_i (t)  &=
    \hat \theta_i(t-1)  \\ &- \delta_i(t)\Sigma_i(t)\xi_i(t) \Big(\hat p_i(t) + \xi_i(t)^\top \hat\theta_i(t-1)\Big),
\end{aligned}
\end{equation}
\begin{equation}
\label{eq:Sigma_i_solution_estimator}
\begin{aligned}
        \Sigma_i(t) &= \Sigma_i(t-1) \\
        & - \delta_i(t)\frac{\Sigma_i(t-1)\xi_i(t)\xi_i(t)^\top\Sigma_i(t-1)}
      {1 + \delta_i(t)\,\xi_i(t)^\top \Sigma_i(t-1)\,\xi_i(t)}.
\end{aligned}
\end{equation} 
\end{subequations}
where we denote $\xi_i(t) = \nabla\Phi(x_i(t))$, $\delta_i(t) = \mathbf{1} \{x_i(t) \in \interior \mathcal{X}_i\}$ and $\mathbf{1}\{\cdot\}$ is the indicator function. 
Denote $\hat \theta(t) = (\hat \theta_1(t)^\top , \ldots, \hat \theta_n(t)^\top)^\top$.
Observe that~\eqref{eq:Sigma_i_solution_estimator} represents a rank-one update on the information matrix $\Sigma_i(t)^{-1}$, that is,
\begin{equation}
\label{eq:Sigma_i_rank_one}
    \Sigma_i(t)^{-1} = \Sigma_i(t-1)^{-1} + \xi_i(t)\xi_i(t)^\top \delta_i(t).\tag{\ref{eq:Sigma_i_solution_estimator}'}
\end{equation}

To simplify notation, we will denote the regressor vectors as $\xi_i(t) = \nabla \Phi(x_i(t))$ and define $\lambda_i(t) = \lambda_{\min}(\Sigma_i(t)^{-1})$.

\subsection{Convergence of the Type Estimator}
 In this subsection, we establish a sufficient condition for the strong consistency of  \eqref{eq:solution_estimator} with respect to the spectrum of the information matrix $\Sigma_i(t)^{-1}$.

\begin{theorem}
\label{thm:sufficient_strong_consistency}
    Suppose Assumptions~\ref{ass:types} and \ref{ass:noise} hold. 
    For each $i\in [n]$, the estimator $\hat \theta_i(t)$ given in
    \eqref{eq:solution_estimator} satisfies
    \[\|\hat \theta_i(t) - \theta_i^*\|_2^2 = O\left(\lambda_i(t)^{-1} \log t\right) \text{ a.s.},\]
    where $\theta_i^*$ is the agent's true type. 
    Consequently, $ \hat \theta_i(t) \rightarrow\theta_i^*$ a.s. if 
    $\log t = o(\lambda_{i}(t))$ a.s.
\end{theorem}
\begin{proof}
    Let $\mathcal{F}_i(t) = \sigma(e_i(s): s\leq t)$ denote the filtration generated by $e_i(t)$. Under Assumption~\ref{ass:noise}, $e_i(t)$ is a martingale difference sequence and satisfies
    \(\sup_{t\geq 1}\mathbb E[|e_i(t)|^a \mid \mathcal{F}_i(t-1)] < \infty, \text{ for }a>2.\)
    By~\cite[Theorem~1]{laiLeastSquaresEstimates1982} and the observation that $\|\hat \theta_i(t) - \theta_i^*\|_2 = O(\|\hat \theta_i(t) - \theta_i^*\|_\infty)$, it holds that
    \[\|\hat \theta_i (t) -  \theta_i^*\|_2 = O \Big(\sqrt{\frac{\log(\lambda_{\max}(t))}{\lambda_i(t)}}\Big) \text{ a.s.},\]
    where $\lambda_{\max}(t)$ denotes the maximum eigenvalue of 
    $\Sigma_i(t)^{-1}$. Due to the compactness of  $\mathcal{X}_i$, it also holds that
    \(\lambda_{\max}(t) \leq t \max_{x \in \mathcal{X}_i} \|\nabla \Phi(x)\|_2^2 = O(t),\)
    and the result follows.
\end{proof}

The excitation required by~\cite{laiLeastSquaresEstimates1982} is a diminishing excitation condition that guarantees strong consistency in stochastic regression (see~\cite[Example 1]{laiLeastSquaresEstimates1982}).  Since the spectrum of $\Sigma_i(t)^{-1}$ depends on the responses $\{x_i(\tau)\}_{\tau \leq t}$, the growth condition in Theorem~\ref{thm:sufficient_strong_consistency} can be guaranteed indirectly by adding appropriate noise in incentives $p_i(t)$. The subsection that follows shows that independent, normally-distributed incentives satisfy this requirement.

\subsection{Excitation under Normally Distributed Incentives}

In regression \eqref{eq:model}, the noise $e_i(t)$ does not affect the regressor sequence by Assumption~\ref{ass:noise}, so 
the planner must explicitly inject excitation via $p_i(t)$. Theorem~\ref{thm:information_growth} proves that an i.i.d. Gaussian probing policy suffices to obtain linear growth in the spectrum of the information matrix.

\begin{theorem}
\label{thm:information_growth}
    Suppose Assumptions~\ref{ass:types} and \ref{ass:noise} hold. For each agent $i \in [n]$, the i.i.d incentive policy $\{p_i(\tau)\}_{\tau \leq t}$ with each $p_i(\tau)\sim \mathcal{N}(0,\sigma^2)$
    guarantees that 
    \(\lambda_i(t) = \Theta(t) \text{ a.s.}\) 
\end{theorem}
\begin{proof}
    Presented in Section~\ref{sec:proofs}.
\end{proof}

An intermediate result of independent interest is that i.i.d. Gaussian incentives are persistently exciting for the kernel regression~\eqref{eq:model}, despite the nonlinearity introduced by the mapping $\nabla \Phi(\cdot)$. Crucially, this excitation holds despite individual incentives not necessarily belonging to the informative region \(\mathcal P_i\) for any agent $i\in[n]$.

\begin{lemma}
\label{lemma:excitation}
    Under Assumptions~\ref{ass:types} and \ref{ass:noise}, for each agent $i\in [n]$ and i.i.d incentives $p_i(t)$ distributed according to $ \mathcal{N}(0,\sigma^2)$,
    there exists
    $\delta > 0 $ such that 
    \[  \mathbb E\left[ \xi_i(t)\xi_i(t)^\top \delta_i(t)\right] \succeq \delta \mathbb I.\]
\end{lemma}

\begin{proof}
    Presented in Section~\ref{sec:proofs}.
\end{proof}

Theorem~\ref{thm:information_growth} establishes that i.i.d Gaussian incentives provide sufficient excitation to the parameter estimator given in \eqref{eq:solution_estimator}. This will be utilized in the next section to develop an adaptive incentive policy that guarantees strong consistency.

\section{An Algorithm for Regret-Minimizing\\Adaptive Incentive Design}
\label{sec:algorithm}
In this section, we leverage the convergence of the estimator in 
Section~\ref{sec:type_estimation} to design an adaptive incentive policy that achieves both objectives of RAID. The proposed mechanism, given in 
Algorithm~\ref{alg:one}, alternates between exploration and exploitation phases to guarantee vanishing average regret.

\SetKwComment{Comment}{/* }{ */}
\SetKwInOut{KwParameters}{Parameters}
\SetKwInOut{KwInit}{Initialize}
\SetKw{KwAnd}{and}

\begin{algorithm}
\caption{Regret--minimizing Adaptive Incentive Design (RAID)}\label{alg:one}
  \KwParameters{%
    $\gamma \in [\frac{2}{3}, 1)$, $\sigma^2 > 0$, 
    $t\geq 1$.
  }
\KwOut{Estimates $\{ \hat\theta_i(\tau)\}_{\tau\leq t}$}
Define $A(\tau) = \tau^{\gamma}\log \tau$\;
Initialize $\Sigma_i(0) \succ 0$ and $\hat \theta_i(0)$, $i\in [n]$\;
\For{$i \in [n]$ \KwAnd $\tau = 1, \ldots, t$}{
    \Comment{Exploration phase}
    \uIf{$\tr(\Sigma_i(\tau-1)) > A(\tau-1)^{-1}$}{
        Issue $p_i(\tau) \sim \mathcal N(0,\sigma^2)$\;
        Observe $x_i(\tau) = x_i^*(p_i(\tau))$\;
        Update $\Sigma_i(\tau)$ and $\hat \theta_i(\tau)$, according to \eqref{eq:solution_estimator}\; 
    }
    \Comment{Exploitation phase}
    \uElseIf{$\tr(\Sigma_i(\tau-1)) \leq A(\tau-1)^{-1}$}{
        Issue
        $p_i(\tau) \gets  -\hat\theta_i(\tau-1)^\top \nabla \Phi(x_{i,\des})$\;
        Maintain $\Sigma_i(\tau) \gets \Sigma_i(\tau-1)$ and $\hat \theta_i(\tau) \gets \hat \theta_i(\tau-1)$
    }
}
\end{algorithm}

For each $i\in [n]$, the equation $p_{i,\des} = - \theta_i^{*\top} \nabla \Phi(x_{i,\des})$ defines the unique incentive that would elicit the desired response $x_{i,\des}$. Without prior knowledge of types, it is natural for the system planner to substitute  the estimate $\hat \theta_i(t)$ for 
$\theta_i^*$, which defines the adaptive policy \eqref{eq:self_tuning_regulator}. 
By Theorem~\ref{thm:sufficient_strong_consistency}, the algorithm must ensure that $\log t = o(\lambda_i(t))$ a.s. to achieve strongly-consistent estimation of $\theta_i^*$.
For this reason, the proposed algorithm uses a threshold-based switching rule to alternate between exploration and exploitation phases. 
Let
\begin{equation}
\label{eq:control}
    p_i(t)\! =\! \begin{cases}
        -\hat \theta_i(t\!-\!1)^\top \nabla\Phi(x_{i,\des}), \!\!\!\! & \tau_i(k) \!\leq\! t\! < \! \sigma_i(k)\\
        \epsilon_i(t), & \sigma_i(k) \! \leq \! t \! < \! \tau_i({k\!+\!1})
   \end{cases}
\end{equation}
where the probing sequence $\{\epsilon_i(t)\}_{t}$ is i.i.d with each $\epsilon_i(t) \sim \mathcal{N}(0, \sigma^2)$, and the switching schedule $\{\tau_i(k)\}_{k\geq 1}$ and $\{\sigma_i(k)\}_{k\geq 1}$ is designed as follows. Let $\tau_i(1) = 0$, 
\begin{equation}
\label{eq:switches}
\begin{aligned}
    \sigma_i(k) & = \inf\{t : t\geq \tau_i(k), \ \tr(\Sigma_i(t)) > A(t)^{-1}\}, \\
    \tau_i({k+1}) & = \inf\{t : t\geq \sigma_i(k), \ \tr(\Sigma_i(t)) \leq A(t)^{-1}\}, \\
\end{aligned}
\end{equation}
for $k\geq 1$, where $A(t)$ is given in Algorithm~\ref{alg:one}.
 
Due to the switching schedule \eqref{eq:switches}, the \emph{exploitation} phase takes place during iterations $t \in [\tau_i(k), \sigma_i(k))$ for all $k\geq 1$, and utilizes the current type estimate $\hat \theta_i(t)$ to regulate the agent's behavior. During exploitation, the planner does not update the estimate $\hat \theta_i(t)$. On the other hand, \emph{exploration} occurs when the information matrix $\Sigma_i(t)^{-1}$ needs to be excited. Notice that $1/\lambda_{i}(\Sigma_i(t)^{-1}) \leq \tr(\Sigma_i(t)) \leq d/ \lambda_{i}(\Sigma_i(t)^{-1})$, and therefore, when $\lambda_{i}(t) \leq A(t)$, it follows that $\tr(\Sigma_i(t)) \geq A(t)^{-1}$.  This condition defines the exploration iterations $t \in [\sigma_i(k), \tau_i({k+1}))$ for all $k\geq 1$, during which the planner utilizes \eqref{eq:solution_estimator} to update the estimate $\hat \theta_i(t)$ to better approximate $\theta_i^*$.

Applying Algorithm~\ref{alg:one} with the incentive policy \eqref{eq:control} and switching schedule \eqref{eq:switches}, the planner can guarantee the a.s. convergence rate of $\hat \theta_i(t)$ to $\theta_i^*$, for each $i\in [n]$, and that the $t$-stage average regret $t^{-1}R_t$ a.s. vanishes. This constitutes the main result of this work, and is presented below.

\begin{theorem}
\label{thm:algorithm_regret}
    Suppose Assumptions~\ref{ass:feasibility}, \ref{ass:types}, and \ref{ass:noise} hold. For every $i\in [n]$ and $\gamma \in [\tfrac{2}{3},1)$, the type estimates $\{\hat \theta_i(\tau)\}_{\tau \leq t}$ produced by Algorithm~\ref{alg:one} satisfy
    \[\| \hat \theta_i(t) - \theta_i^*\|_2^2 = O(t^{-\gamma}), \text{ a.s.}\]
    Moreover, the system planner's regret $ R_t$, as defined in \eqref{eq:regret_definition}, satisfies
    \[ R_t = O (t^\gamma \log t), 
    \text{ a.s.}\]
\end{theorem}
\begin{proof}
    Presented in Section~\ref{sec:proofs}.
\end{proof}
The above regret bound is sublinear for every choice of $\gamma \in [\frac{2}{3},1)$, hence the policy is asymptotically optimal. Among admissible choices, $\gamma=\frac{2}{3}$ yields the fastest decay of $t^{-1}R_t$ and thus represents the planner’s preferred parameter.

\section{Numerical Examples}
\label{sec:numerical_examples}
We illustrate the performance of the proposed algorithm through numerical simulations. 
We first consider a game of $n=3$ players
with cubic cost functions according to \eqref{eq:costs} and the feasible response set $\mathcal{X} = [-1, 1]^3$. Player preferences are represented by the unknown types $\theta_1^* = (1, 0.5, 0)^\top$, $\theta_2^* = (0, 3.5, 1)^\top$, and $\theta_3^* = (-1, 3.5, -1)^\top$, and each agent model~\eqref{eq:model} is perturbed by the process $\{e_i(t)\}_{t\geq 1}$ which is i.i.d. and normally-distributed with variance $0.1$.
The system planner has a desired response $x_\des = (0.5, 0.5, 0.5)^\top$, and utilizes Algorithm~\ref{alg:one} with parameters $\gamma = \tfrac{2}{3}$ and $\sigma^2 = 2$.

Theorem~\ref{thm:algorithm_regret} predicts that the parameter estimation error $\| \hat \theta_i(t)- \theta_i^*\|_2$ decays as  $O (t^{-\gamma/2})$ a.s. and the average regret $t^{-1}R_t$ decays as $O(t^{\gamma-1}\log t)$ a.s. Taking $100$ independent runs of Algorithm~\ref{alg:one}, Figure~\ref{fig:algorithm_estimation_error} depicts the planner's empirical expectation of the estimation errors and accumulated regret, respectively. The results obtained are consistent with the predicted almost-sure decay rates of the agents' type estimation error and the planner's average regret.
\begin{figure}
    \centering
    \includegraphics[width=\linewidth]{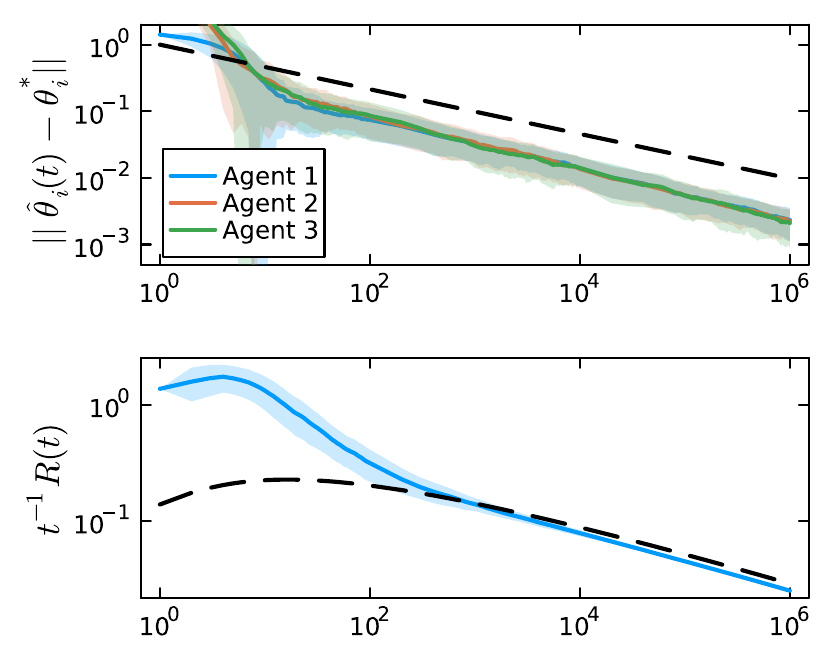}
    \caption{Mean (solid) and $\pm 1$ standard deviation (shaded) of $\|\hat \theta_i(t)- \theta_i^*\|_2$ and $t^{-1}R_t$  over 100 independent realizations of Algorithm~\ref{alg:one}. Dashed lines indicate the a.s. convergence rates predicted in Theorem~\ref{thm:algorithm_regret}. }
    \label{fig:algorithm_estimation_error}
\end{figure}

Figure~\ref{fig:different_noise} illustrates that the convergence of the parameter estimator is dependent on the injected probing excitation rather than the model noise $e_i(t)$. In particular, the
convergence rate in Theorem~\ref{thm:algorithm_regret} is guaranteed even when $e_i(t)$ is of measure-zero support, such as a Rademacher-distributed random variable on $\pm0.1$ (Fig. \ref{fig:different_noise}, second row). This confirms that the weak excitation condition of Theorem~\ref{thm:sufficient_strong_consistency} is satisfied by the probing policy of Algoritm~\ref{alg:one}.

\begin{figure}
    \centering
    \includegraphics[width=\linewidth]{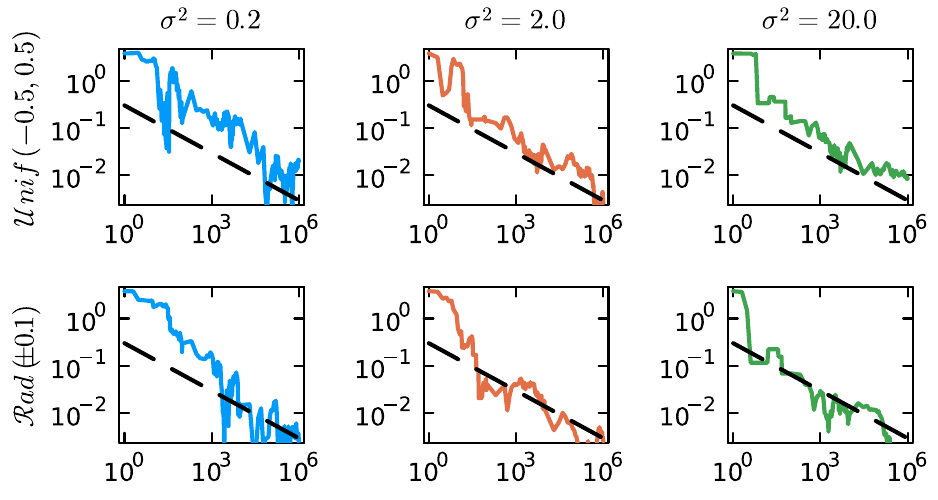}
    \caption{Estimation error $\|\hat{ \theta}_i(t)- \theta_i^*\|_2$ over a single realization of Algorithm~\ref{alg:one} for a single agent. Dashed lines indicate the rate $O(t^{-\gamma/2})$, $\gamma= 2/3$. Columns vary the probing parameter $\sigma^2$.
    (Top row) $e_i(t)\sim \mathcal{U}nif [-0.5, 0.5]$.
    (Bottom row) $e_i(t)$ is a Rademacher-distributed random variable on $\{\pm 0.1\}$.}
    \label{fig:different_noise}
\end{figure}

Figure~\ref{fig:paramters} validates the regret bound in Theorem~\ref{thm:algorithm_regret} across different choices of parameters in Algorithm~\ref{alg:one}. As predicted, $\gamma = \tfrac{2}{3}$
yields the fastest decay of $t^{-1}R_t$ among admissible values and is therefore the planner's preferred choice. Observe further that the probing variance $\sigma^2 $ affects the onset of the asymptotic regime in Algorithm~\ref{alg:one}. When $\sigma^2$ is small (e.g $\sigma^2 = 0.5$), the information matrix grows slowly, resulting in near-constant average regret over a given simulation horizon. This suggests that the probing variance and parameter $\gamma$ should be chosen relative to the problem horizon, and that a finite-time analysis of the algorithm is needed to characterize this tuning. 
This is left as an interesting direction for future work.

\begin{figure}
    \centering
    \includegraphics[width=\linewidth]{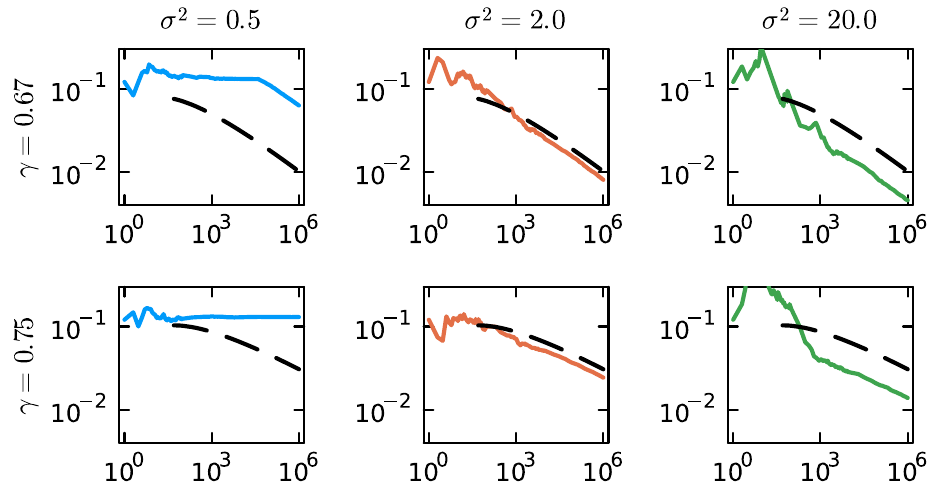}
    \caption{Average regret $t^{-1}R_t$ over a single realization of Algorithm~\ref{alg:one} for a single agent. Model noise is $e_i(t) \sim \mathcal{N}(0, 0.1)$. Dashed lines indicate the rate $O(t^{\gamma-1}\log t)$. Columns vary the probing parameter $\sigma^2$ and rows vary the switching parameter $\gamma$.}
    \label{fig:paramters}
\end{figure}

\section{Proofs}
\label{sec:proofs}
\subsection{Proof of Lemma~\ref{lemma:excitation}}

\begin{lemma}
\label{lemma:PD_monomial_moments}
Let $f: \mathbb R \to \mathbb R^d$ be the vector-valued mapping  $f(x)= (1, x, \ldots, x^{d-1})^\top$, and let
$x \sim \mathcal{N}(0, \sigma^2)$. For any nonempty, open interval $A \subset \mathbb R$,
\[\mathbb E [f(x)f(x)^\top \mid x \in A] \succ 0.\]
\end{lemma}
\begin{proof}
Let $A = (a,b)$ for $a<b$. Then
\[\begin{aligned}
    \mathbb E [f(x)f(x)^\top\!\!\mid\! x \in A] = \frac{\sigma^{-1}}{\Phi(b)\!-\!\Phi(a)} \! \int_a^b \!\!f(x)f(x)^\top \!\phi(\frac{x}{\sigma}) dx,
\end{aligned}\]
where $\Phi$ and $\phi$ represent the CDF and PDF of $\mathcal{N}(0,1)$, respectively. 
For any
$z \in \mathbb R^d$ satisfying
$z^\top\mathbb E [f(x)f(x)^\top \mid x \in A]z = 0$,
it holds that
\begin{align*}
    \int_a^b (f(x)^\top z)^2 \phi\Big(\frac{x}{\sigma}\Big) dx = 0.
\end{align*}
This further implies that the polynomial
\begin{align*}
    T_z(x) = \sum_{k=0}^{d-1} z_k x^k = 0 \text{ a.s. on the event }x\in A,
\end{align*}
where $z_k$ is the $k$-th element of $z$.
However, $T_z(x)$ is a polynomial of degree $d-1$, and is $0$ on a set of non-zero measure if and only if $z=0$.
\end{proof}

\begin{proof}[\textbf{Proof of Lemma~\ref{lemma:excitation}}]
Notice that, due to Lemma~\ref{lemma:NE_bijection_mapping}, 
\[
\begin{aligned}
& \mathbb E[\xi_i(t)\xi_i(t)^\top \delta_i(t)]\\  
= \, &\mathbb P (x_i(t) \in \interior \mathcal{X}_i )\mathbb E[\xi_i(t)\xi_i(t)^\top \mid x_i(t) \in \interior \mathcal{X}_i ] \\ 
= \,& 
\mathbb P (p_i(t) \in \mathcal{P}_i )\mathbb E[\xi_i(t)\xi_i(t)^\top \mid p_i(t) \in \mathcal{P}_i ],
\end{aligned}\]
so it is sufficient to prove $\mathbb E[\xi_i(t)\xi_i(t)^\top \mid p_i(t) \in \mathcal{P}_i ] \succ 0$.

First, for the vector-valued map
$f(x)= (1, x, \ldots, x^{d-1})^\top$, it holds that $\xi_i(t) = \nabla\Phi_i(x_i(t)) = D_1f(x_i(t))$, where $D_1 = \operatorname{diag}(1, 2, \ldots, d)$.
Second, on the event that $p_i(t) \in \mathcal{P}_i$, Lemma~\ref{lemma:NE_bijection_mapping} guarantees that $x_i(t) = h_i(p_i(t))$ for an $h_i\in \mathcal{C}^1$. By the mean value theorem, there exists $w_i \in (0, p_i(t))$ such that $x_i(t) = h(0) + h'_i(w_i)p_i(t)$. By the binomial theorem on $f(x_i(t))$, we have 
\[\begin{aligned}
    & f(x_i(t))f(x_i(t))^\top \\ = \, &\Pi\left(h_i(0)\right) f\left(h'_i(w_i)p_i(t)\right) f\left(h'_i(w_i)p_i(t)\right)^\top\Pi\left(h_i(0)\right)^\top  \\
     =\, &\Pi\left(h_i(0)\right) D_2 f\left(p_i(t)\right) f\left(p_i(t)\right)^\top D_2\Pi\left(h_i(0)\right)^\top ,
\end{aligned} \]
where $D_2 = \operatorname{diag}(1, h'_i(w_i), \ldots, h_i'(w_i)^{d-1})$, and $\Pi(x) \in \mathbb R^{d\times d}$ is a nonsingular lower-triangular matrix.
Taking the expectation and using Lemma~\ref{lemma:PD_monomial_moments}, the result follows.
\end{proof}

\subsection{Proof of Theorem~\ref{thm:information_growth}}

\begin{proof}
    By the compactness of $\mathcal{X}_i$, $\|\xi_i(t)\|_2$ is bounded and $\mathbb E \|\xi_i(t)\|_2 < \infty$.  
    Define the matrix  $M=\mathbb E[\xi_i(t) \xi_i(t)^\top \delta_i(t)]$, where $M \succ 0$ due to Lemma~\ref{lemma:excitation}. 
    By the strong law of large numbers
    \[\begin{aligned}
        t^{-1}\Sigma_i(t)^{-1} = t^{-1} \sum_{\tau = 1}^t \xi_i(t) \xi_i(t)^\top \delta_i(t)  \overset{a.s.}{\longrightarrow} M.
    \end{aligned} \]
    Applying Weyl's inequality,
    it holds that
    \[|t^{-1} \lambda_{\min}(\Sigma_i(t)^{-1}) - \lambda_{\min}(M) | \leq \|t^{-1}\Sigma_i(t)^{-1}- M\|_2,\]
    which in turn implies
    $t^{-1}\lambda_{\min}(\Sigma_i(t)^{-1}) \overset{a.s.}{\longrightarrow} \lambda_{min}(M)$, and hence, $\lambda_i(t) = \Theta(t)$ a.s. 
\end{proof}

\subsection{Proof of Theorem~\ref{thm:algorithm_regret}}
For each $i\in [n]$ and $t$, consider a schedule $\{\tau_i(k)\}_{k\geq 1}$, $\{\sigma_i(k)\}_{k\geq 1}$ according to \eqref{eq:switches}.
Let $ K_{i,t}= \sup\{k : \tau_i(k) \leq t\}$. By definition, it holds that $\tau_i(K_{i,t}) \leq t <  \tau_i(K_{i,t}+1)$, for any $t$.
Define $\#_i(t)$ to be the total number of exploration samples up to time $t$, 
that is,
\[\#_i(t) = \sum_{k=1}^{K_{i,t}-1} (\tau_i(k+1) - \sigma_i(k)) + \max\{0, t - \sigma_i(K_{i,t})\}.\]

\begin{lemma}
\label{lemma:number_of_exploration_samples}
Suppose Assumptions~\ref{ass:types} and \ref{ass:noise} hold. For every $i\in[n]$, Algorithm~\ref{alg:one} guarantees that $\lambda_{i}(t) = \Theta(A(t))$ a.s., and $\#_i(t) = O(A(t))$ a.s.
\end{lemma}
\begin{proof}
    
    We will denote $\lambda_i(t) = \lambda_{\min}(\Sigma_i(t)^{-1})$ and $\operatorname{tr}_i(t) = \operatorname{tr}(\Sigma_i(t))$. The proof holds for each $i\in[n]$, so we omit the subscript $i$ to ease notation.
    
    Notice that $\lim_{t\to\infty} \#(t) = \infty$ a.s., otherwise, there would exist some $t_0$ such that $\operatorname{tr}(t_0) \leq A(t)^{-1}$ for all $t\geq t_0$, which contradicts the monotonicity of $A(t)$.

    Denote by $\mathcal{I}(t)$ the times up to $t$ that belong to exploration phases, hence $|\mathcal{I}(t)| = \#(t)$. 
    The subsequence of incentives satisfies that $\{p(\tau)\}_{\tau \in \mathcal{I}(t)} = \{\epsilon(\tau)\}_{\tau \in \mathcal{I}(t)}$, where $\{\epsilon(\tau)\}_{\tau \in \mathcal{I}(t)}$ is an i.i.d. sequence of Gaussian noise. 
    To see this, we can consider an auxiliary noise sequence $\{\mu(k)\}_{k\geq 1}$, where $\mu(k)$ has the same distribution as $\epsilon(t)$, i.e.,  $\mu(k) \sim \mathcal{N}(0,\sigma^2)$. When $\tau \in \mathcal I(t)$, sampling the probing noise $\epsilon(\tau)$ is equivalent to sampling from the process $\{\mu(k)\}$ with $\epsilon({\tau})=\mu(\#(\tau))$. Therefore, 
    $\{\epsilon(\tau)\}_{\tau \in \mathcal{I}(t)}= \{\mu(k)\}_{k=1}^{\#(t)}$ is an i.i.d. probing sequence.
    
    By Theorem~\ref{thm:information_growth}, there exists some constant $c>0$ such that 
    \[\lim_{t\rightarrow \infty} \frac{1}{\#(t)}\lambda_{\min} \Big(\sum_{\tau \in \mathcal{I}(t)}\xi(\tau)\xi(\tau)^\top \delta(\tau) \Big) = c,  \text{ a.s.}\]
    $\Sigma_i(\tau)^{-1}$ is only updated in Algorithm~\ref{alg:one} when $\tau \in \mathcal{I}(t)$, so
   $ \sum_{\tau \in \mathcal{I}(t)} \xi(\tau)\xi(\tau)^\top \delta(\tau) = \Sigma_i(t)^{-1}$.
    Therefore, for any $\epsilon > 0$, there exists some $t_\epsilon>0$ such that
    \begin{equation}
    \label{eq:thm3_proof_1}
           (c- \epsilon) \#(t) \leq \lambda(t)  \leq (c+\epsilon) \#(t), \quad \forall t\geq t_\epsilon.
    \end{equation}
    Denote $c^- = c-\epsilon$, and $c^+ = c+\epsilon$.
    Moreover, denote $B = \max_{x\in \mathcal{X}} \|\nabla \Phi(x)\|^2$, such that for all $t$ 
    \begin{equation}
    \label{eq:thm3_proof_2}
    \begin{aligned}
        \lambda(t+1) & \leq \lambda(t) + \lambda_{max}(\xi(t) \xi(t)^\top)  \leq  \lambda(t) +B.
    \end{aligned}
    \end{equation}
    Finally, by~\eqref{eq:switches}, there exists some $t_0>0$ such that
    \begin{equation}
    \label{eq:thm3_proof_3}
    \begin{aligned}
        \lambda(t) \geq  A(t), \quad &\forall t\geq t_0 : \    \tau(K_t) \leq t < \sigma(K_t), \\ 
        \lambda(t) \leq d A(t), \quad &\forall t\geq t_0: \  \sigma(K_t) \leq t<\tau(K_t+1), 
    \end{aligned}
    \end{equation}
    where $d$ is the dimension of the codomain of $\Phi$ in \eqref{eq:parametrization}. Observe that \eqref{eq:thm3_proof_2}-\eqref{eq:thm3_proof_3} ensure $\lambda(t) = O(A( t))$. Specifically,
    \begin{enumerate}
        \item when $\sigma({K_t}) \leq t <  \tau(K_t+1)$, 
        it holds \(\lambda(t) \leq  dA(t)\);
        \item when $t = \tau(K_t)$, by \eqref{eq:thm3_proof_2} and \eqref{eq:thm3_proof_3}, 
        \[\lambda(t) \leq \lambda(t-1) + B  \leq dA(t-1) + B; \]
        \item when $\tau(K_t) < t < \sigma(K_t)$, by \eqref{eq:thm3_proof_2}, 
        \[\begin{aligned}
            \lambda(t) = \lambda(\tau(K_t)) & \leq dA(\tau(K_t)-1) + B  \\&\leq dA(t-1) + B.
        \end{aligned}\]
    \end{enumerate}
    Moreover, \eqref{eq:thm3_proof_1}-\eqref{eq:thm3_proof_3} ensure $\lambda(t) = \Omega(A( t))$ a.s., because
    \begin{enumerate}
        \item when $\tau(K_t) \leq t < \sigma(K_t)$, \(\lambda(t) \geq  A(t)\) due to  \eqref{eq:thm3_proof_3}; 
        \item when $\sigma(K_t) \leq t < \tau(K_t+1) $, by \eqref{eq:thm3_proof_1} we have, with probability one,
        \[\begin{aligned}
            \lambda(t) & \geq c^- \left( \#(\sigma(K_t)-1) + t-\sigma(K_t)+1 \right) \\
            & \geq \frac{c^-}{c^+}\left( \lambda(\sigma(K_t)-1) + c^+(t-\sigma(K_t)+1) \right) \\
            & \geq \frac{c^-}{c^+}\left( A(\sigma(K_t)-1) + c^+(t-\sigma(K_t)+1) \right).
        \end{aligned} \]
        Algorithm~\ref{alg:one} selects $A(t)= o(t)$, so the second term grows strictly faster than $A(t)$. Therefore, it holds that $\lambda(t)\geq \min\{\frac{c^-}{c^+}, c^+\} \cdot A(t) $, and $\lambda(t) = \Omega(A(t))$ a.s.
    \end{enumerate}
     
    We conclude that $\lambda(t) = \Theta(A(t))$ a.s., and the upper bound on the growth of $\#(t)$ follows from \eqref{eq:thm3_proof_1}.
\end{proof}

\begin{proof}[\textbf{Proof of Theorem~\ref{thm:algorithm_regret}}]
    By Lemma~\ref{lemma:number_of_exploration_samples}, $\lambda_i(t) = \Theta(A(t))$
    a.s. for each $i\in [n]$. Then, by
    Theorem~\ref{thm:sufficient_strong_consistency},
    \[\|\hat \theta_i(t) - \theta_i^*\|_2^2 = O\Big(\frac{\log t}{\lambda_i(t)}\Big)= O\Big(\frac{\log t}{A(t)}\Big) = O(t^{-\gamma}), \text{ a.s.}\]

    The system planner's regret is
    \({R}_t = \sum_{i\in [n]} \sum_{\ell=1}^t |x_i(\ell)- x_{i,\des}|^2.\)
    For any $i \in [n]$, observe that
    \begin{align}
        \sum_{\ell=1}^t |x_i(\ell)- x_{i,\des}|^2\leq 
          &  \sum_{k=1}^{K_{i,t}} \Big(\sum_{\ell=\tau_i(k)}^{\sigma_i(k)-1} |x_i(\ell)- x_{i,\des}|^2 \tag{$i$}
        \\& + \sum_{\ell=\sigma_i(k)}^{\tau_i(k+1)-1}  |x_i(\ell)- x_{i,\des}|^2  \Big) . \tag{$ii$}
    \end{align}
    We consider each of the RHS terms above separately. 
    First, term $(ii)$ satisfies
    \[\begin{aligned}
        (ii)  \leq C\sum_{k=1}^{K_{i,t}} \sum_{\ell=\sigma_i(k)}^{\tau_i(k+1)-1}\!\! 1= O(\#_i(t)) = O(t^\gamma \log t) \text{ a.s.},
    \end{aligned}\]
    where $C= \max_{x\in \mathcal{X}_i}|x -  x_{i,\des}|^2$, and the last step follows from Lemma~\ref{lemma:number_of_exploration_samples}.
    Second, term $(i)$ satisfies
    \[\begin{aligned}
    \sum_{k=1}^{K_{i,t}}\sum_{\ell=\tau_i(k)}^{\sigma_i(k)-1} \!\! |x_i(\ell)- x_{i,\des}|^2 \leq  \frac{1}{m_i^2} \sum_{k=1}^{K_{i,t}}\sum_{\ell=\tau_i(k)}^{\sigma_i(k)-1} \!\! |p_i(\ell)- p_{i,\des}|^2 \\ \leq \frac{\|\nabla\Phi(x_{i,\des})\|_2^2}{m_i^2} \sum_{k=1}^{K_{i,t}}\sum_{\ell=\tau_i(k)}^{\sigma_i(k)-1}\!\! \|\hat \theta_i(\ell)- \theta_{i}^*\|_2^2.
    \end{aligned} \]
    We have established that $\|\hat \theta_i(\ell)- \theta_i^*\|_2^2 = O(\ell^{-\gamma})$ a.s., hence $\sum_{\ell=1}^t\|\hat \theta_i(\ell)- \theta_i^*\|_2^2 = O(t^{1-\gamma})$ a.s.
    For all $\gamma \in [\tfrac{2}{3}, 1)$, observe that term $(ii)$ increases faster than term $(i)$, and it follows that ${R}_t = O(t^\gamma \log t)$ a.s.
\end{proof}

\section{Concluding Remarks}
\label{sec:conclusion}

This work introduces the Regret-Minimizing Adaptive Incentive Design (RAID) problem, a unified framework for type identification and control in adaptive incentive design. 
Our first contribution is the design of a switching incentive policy that alternates between probing (exploration) and estimate-based (exploitation) incentives. We prove that this policy asymptotically minimizes the planner’s average 
$t$-stage regret almost surely.
Our second contribution is to establish the strong consistency of the type
estimator under a weak excitation condition for least-squares estimation,
thereby relaxing the standard persistence-of-excitation assumptions used in adaptive incentive design.

Throughout this paper, we have also identified two directions for future extensions of the RAID framework. First, and most significant, is to relax the decoupled-costs assumption in the nominal cost structure of \eqref{eq:costs}, and generalize RAID to settings with multi-agent interactions among participants.
Second, we aim to pursue a formal treatment of the Error-in-Variables problem that arises in regression \eqref{eq:model} when the observation noise is not independent of the agents' responses. 

\bibliographystyle{IEEEtran} 
\bibliography{IEEEabrv, bibliography}
\end{document}